\documentclass[12pt,reqno]{article}

\usepackage[usenames]{color}
\usepackage{amssymb}
\usepackage{graphicx}
\usepackage{amscd}
\usepackage{commath}

\usepackage[colorlinks=true,
linkcolor=webgreen,
filecolor=webbrown,
citecolor=webgreen]{hyperref}

\definecolor{webgreen}{rgb}{0,.5,0}
\definecolor{webbrown}{rgb}{.6,0,0}

\usepackage{color}
\usepackage{fullpage}
\usepackage{float}

\usepackage{graphics,amsmath,amssymb}
\usepackage{amsthm}
\usepackage{amsfonts}
\usepackage{latexsym}
\usepackage{epsf}

\usepackage{mathrsfs}

\newcommand{\seqnum}[1]{\href{http://oeis.org/#1}{\underline{#1}}}

\begin{document}


\theoremstyle{plain}
\newtheorem{theorem}{Theorem}
\newtheorem{corollary}[theorem]{Corollary}
\newtheorem{lemma}[theorem]{Lemma}
\newtheorem{proposition}[theorem]{Proposition}

\theoremstyle{definition}
\newtheorem{definition}[theorem]{Definition}
\newtheorem{example}[theorem]{Example}
\newtheorem{conjecture}[theorem]{Conjecture}
\newtheorem{identity}{Identity}

\theoremstyle{remark}
\newtheorem{remark}[theorem]{Remark}

\begin{center}
\vskip 1cm{\LARGE\bf Determinants Containing Powers \\ 
\vskip .12in
of Generalized Fibonacci Numbers}
\vskip 1cm
\large
Aram Tangboonduangjit and Thotsaporn Thanatipanonda\\
Mahidol University International College\\
Nakhonpathom 73170\\
Thailand \\
\href{mailto:aram.tan@mahidol.edu}{\tt aram.tan@mahidol.ac.th} \\
\href{mailto:thotsaporn.tha@mahidol.ac.th}{\tt thotsaporn.tha@mahidol.ac.th}\\
\end{center}

\vskip .2 in

\begin{abstract}
We study the determinants of matrices whose entries are powers of the Fibonacci numbers. We then extend the results to include entries that are powers of  a second-order linear recurrence relation. These results motivate a fundamental identity of determinants whose entries are powers of linear polynomials. Finally, we discuss 
the determinants of matrices whose entries are products 
of the general second-order linear recurrence relations.
\end{abstract}

\section{Introduction}
\label{sec:introduction}

In the first issue of the \textit{Fibonacci Quarterly}, Alfred posed the following problem \cite{alfred-0}:
\begin{quote}
Prove $$\begin{vmatrix}F_n^2 & F_{n+1}^2 & F_{n+2}^2 \\\\ F_{n+1}^2 & F_{n+2}^2 & F_{n+3}^2 \\\\ F_{n+2}^2 & F_{n+3}^2 & F_{n+4}^2
\end{vmatrix} = 2(-1)^{n+1},$$ where $F_n$ is the $n$th Fibonacci number.
\end{quote}
In the second volume  of the \textit{Fibonacci Quarterly}, Parker \cite{parker} posed a similar problem with 
the exponent of each entry changed to $4$ and the dimension of the matrix changed to $5\times 5$.
These two results naturally suggest the following
question: what would be the determinant of an analogous form where the dimension of the matrix and the exponent of each entry are arbitrary? Carlitz \cite{carlitz} answered this question by showing that the determinant of the form $D(r,n)=\abs{F_{n+i+j}^r}$, where $i,j=0,1,\dots,r$, is given by

\begin{equation}\label{eq1} 
D(r,n)=(-1)^{(n+1){r+1 \choose 2}}(F_1^r F_2^{r-1}\cdots F_r)^2\cdot\prod_{i=0}^r {r\choose i}.
\end{equation}

In this paper we generalize the entries even further by considering the determinant of the form $D(r,s,k,n)=\abs{F_{s+k(n+i+j)}^r}$, where $i,j=0,1,\dots,r$.  We show that
\begin{equation}\label{eq2}
D(r,s,k,n)=(-1)^{(s+kn+1) {r+1 \choose 2}}(F_k^rF_{2k}^{r-1}\cdots F_{rk} )^2\cdot\prod_{i=0}^r \binom{r}{i}.
\end{equation}
Carlitz based his proof of formula \eqref{eq1} on the Binet form of the Fibonacci numbers; whereas we employ the matrix methods, such as the factorization method of Krattenthaler, to prove formula \eqref{eq2}. In addition, we require a generalized form of the Catalan identity, proved by Melham and Shannon \cite{melham}. In Section 2, we present an alternative proof of the generalized Catalan identity using a matrix representation of the sequence and the properties of the matrix multiplication.  In Section 3, we then present the proof of formula \eqref{eq2} as a special case of the determinant with entries involving the powers of the numbers in a second-order linear recurrence with constant coefficients. In the last section, we present the determinant whose entries are the products of the numbers defined as a second-order linear recurrence with constant coefficients using the Desnanot-Jacobi identity. The methodology used for this work relies on a computer programming developed by the second author \cite{thotsaporn}.

\section{Generalized Catalan identity}
\label{sec:Catalan}

The well-known Catalan identity states that for all nonnegative integers $s$ and $i$, 
$$F_{s+i}^2-F_sF_{s+2i}= (-1)^sF_i^2.$$
A generalization of this identity useful for this work is given by Melham and Shannon \cite{melham}. We shall, however, present an alternative proof of this generalization. For integers $a,b,c_1$, and $c_2$ with $c_2\neq 0$, let $W_n = W_n(a,b;c_1,c_2)$ denote the second-order linear recurrence with constant coefficients, defined by $$W_0=a,\, W_1=b \quad\text{and}\quad W_n = c_1W_{n-1}+c_2W_{n-2} \quad\text{for $n\ge 2$.}$$
With this notation, the Fibonacci sequence $(F_n)$ and the Lucas sequence $(U_n)$ correspond to $F_n = W_n(0,1;1,1)$ and  $U_n = W_n(0,1;c_1,c_2)$, respectively.  Moreover, we can use this recurrence to extend the definition of a sequence to the terms with negative indices. Usually, we can explicitly find the relationship between the negative-indexed terms and the positive-indexed terms. For example, for the Fibonacci sequence and the Lucas sequence, we have  
$$F_{-n}= (-1)^{n+1}F_n \quad \text{and} \quad  U_{-n} = (-1)^{n+1}c_2^{-n}U_{n}\quad\text{for $n\ge 1$.}$$

\begin{proposition}[Generalized Catalan Identity]  \label{GenCat}
Let $W_n = W_n(a_0,a_1;c_1,c_2)$ and $Y_n = W_n(b_0,b_1;c_1,c_2)$ be second-order linear recurrences. Then
\begin{equation} \label{eqW}
W_{s+i}Y_{s+j} - W_sY_{s+i+j} = (-c_2)^s(W_1Y_{j} - W_0Y_{j+1}) \cdot U_i,
\end{equation}
for all integers $s,j$, and $i$.
\end{proposition}

\begin{proof}
The proof is by induction on $i$. For the case when $i=0$, the identity is trivial. For the case when $i=1$, we have \\
$$\begin{pmatrix}
  W_{s+1} & Y_{s+j+1}  \\
  W_{s} & Y_{s+j} \\
 \end{pmatrix}  =
\begin{pmatrix}
  c_1 & c_2  \\
  1 & 0 \\
 \end{pmatrix}
\begin{pmatrix}
  W_{s} & Y_{s+j}  \\
  W_{s-1} & Y_{s+j-1} \\
 \end{pmatrix}  
=\begin{pmatrix}
  c_1 & c_2  \\
  1 & 0 \\
 \end{pmatrix}^s
\begin{pmatrix}
  W_{1} & Y_{j+1}  \\
  W_{0} & Y_{j} \\
 \end{pmatrix}  \;\ \text{ for } s \geq 0,$$
 
 and 
$$\begin{pmatrix}
  W_{s+1} & Y_{s+j+1}  \\
  W_{s} & Y_{s+j} \\
 \end{pmatrix}  =
\begin{pmatrix}
  c_1 & c_2  \\
  1 & 0 \\
 \end{pmatrix}^{-1}
\begin{pmatrix}
  W_{s+2} & Y_{s+j+2}  \\
  W_{s+1} & Y_{s+j+1} \\
 \end{pmatrix}  
=\begin{pmatrix}
  c_1 & c_2  \\
  1 & 0 \\
 \end{pmatrix}^s
\begin{pmatrix}
  W_{1} & Y_{j+1}  \\
  W_{0} & Y_{j} \\
 \end{pmatrix}  \;\ \text{ for } s < 0,$$   
where the second equality in both equations follows from repeated application of the matrix representation in the first equality. Taking the determinant of both sides of any one equation yields
\begin{equation} \label{eqWb}
W_{s+1}Y_{s+j}-W_sY_{s+j+1} = (-c_2)^s(W_1Y_j - W_0Y_{j+1}).
\end{equation}
Now, consider two cases.
\begin{enumerate}
\item[]\textbf{Case $i > 1$:} Assume that the identity holds for some integers $i-1$ and $i-2$. Then
\begin{align*} 
W_{s+i}Y_{s+j} - W_sY_{s+i+j} &=
\begin{vmatrix}
  W_{s+i} & Y_{s+i+j}  \\
  W_{s} & Y_{s+j} \\
 \end{vmatrix}\\
&=\begin{vmatrix}
  c_1W_{s+(i-1)}+ c_2W_{s+(i-2)} & c_1Y_{s+(i-1)+j}+ c_2Y_{s+(i-2)+j} \\
  W_{s} & Y_{s+j} \\ 
 \end{vmatrix} \\
&=c_1\begin{vmatrix}
  W_{s+(i-1)} & Y_{s+(i-1)+j}  \\
  W_{s} & Y_{s+j} \\
 \end{vmatrix}
+c_2\begin{vmatrix}
  W_{s+(i-2)} & Y_{s+(i-2)+j}  \\
  W_{s} & Y_{s+j} \\
 \end{vmatrix} \\
&= (-c_2)^s(W_1Y_j - W_0Y_{j+1})  (c_1U_{i-1}+c_2U_{i-2})\\
&= (-c_2)^s(W_1Y_j - W_0Y_{j+1})U_i,
\end{align*}

\item[]\textbf{Case $i < 0$:} Assume that the identity holds for some integers $i+1$ and $i+2$. Then
\begin{align*} 
W_{s+i}Y_{s+j} - W_sY_{s+i+j} &=
\begin{vmatrix}
  W_{s+i} & Y_{s+i+j}  \\
  W_{s} & Y_{s+j} \\
 \end{vmatrix}\\
&=\begin{vmatrix}
  \dfrac{-c_1}{c_2}W_{s+(i+1)}+ \dfrac{1}{c_2}W_{s+(i+2)} 
  & \dfrac{-c_1}{c_2}Y_{s+(i+1)+j}+ \dfrac{1}{c_2}Y_{s+(i+2)+j} \\
  W_{s} & Y_{s+j} \\ 
 \end{vmatrix} \\
&=\dfrac{-c_1}{c_2}\begin{vmatrix}
  W_{s+(i+1)} & Y_{s+(i+1)+j}  \\
  W_{s} & Y_{s+j} \\
 \end{vmatrix}
+\dfrac{1}{c_2}\begin{vmatrix}
  W_{s+(i+2)} & Y_{s+(i+2)+j}  \\
  W_{s} & Y_{s+j} \\
 \end{vmatrix} \\
&= (-c_2)^s(W_1Y_j - W_0Y_{j+1})  (\dfrac{-c_1}{c_2}U_{i+1}+\dfrac{1}{c_2}U_{i+2})\\
&= (-c_2)^s(W_1Y_j - W_0Y_{j+1})U_i,
\end{align*}
\end{enumerate}
where we apply the induction hypothesis in the penultimate equality in both cases. Hence, the proof is complete. 
\end{proof}
We note some special cases of Proposition \ref{GenCat} useful in later sections:
\begin{equation} \label{eqUU}
U_{s+i}U_{s+j}-U_sU_{s+i+j} = (-c_2)^s \cdot U_iU_j,
\end{equation}
\begin{equation} \label{eqWU}
U_{s+i}W_{s+j}-U_sW_{s+i+j} = (-c_2)^s \cdot U_iW_j,
\end{equation}
and
\begin{equation} \label{eqWW}
W_{s+i}W_{s+j}-W_sW_{s+i+j} =  (-c_2)^s \cdot (W_1W_j - W_0W_{j+1})U_i \\
=  (-c_2)^s \cdot \Delta \cdot U_iU_j,
\end{equation}
where $\Delta = \begin{vmatrix}
  W_{1} & W_{2}  \\
  W_{0} & W_{1} \\
 \end{vmatrix} = a_1^2-c_1a_0a_1-c_2a_0^2$.  \\

We justify the second equality of \eqref{eqWW} by applying Proposition \ref{GenCat} as follows: In \eqref{eqW},
let $Y_n=W_n$, substitute $j=1$ and $s=0$, respectively, and rename the index $i$ by $j$. 

The identity \eqref{eqWW} can be restated 
as follows:
\begin{corollary}
Let $k,n,r,s$, and $t$ be integers. Then
\begin{equation} \label{eqWWW}
W_{s+k(n+t)} = A(t)W_{s+kn}+B(t)U_{s+k(n+r)},
\end{equation}
where $A(t) = \dfrac{W_{k(t-r)}}{W_{-kr}}$ and  
$B(t) = \dfrac{-(-c_2)^{-kr}\cdot \Delta \cdot U_{kt}}{W_{-kr}}$. 
\end{corollary}

\begin{proof}
Let integers $k',n',r',s'$, and $t'$ be given. Applying \eqref{eqWW} with $s=-k'r'$, $i=k't'$, and $j=s'+k'(n'+r')$, we have 
\begin{equation*}
W_{-k'r'}W_{s'+k'(n'+t')} = W_{k'(t'-r')}W_{s'+k'n'}-(-c_2)^{-k'r'}
\cdot \Delta \cdot U_{k't'}U_{s'+k'(n'+r')}.
\end{equation*} 
Dividing by $W_{-k'r'}$ on both sides and renaming the variables yield the identity \eqref{eqWWW}, as required.
\end{proof}

\section{Determinants involving powers of terms of second-order recurrence} 
\label{Main}

Our goal in this section is to give the closed form of the determinant of the
$(r+1)\times(r+1)$ matrix whose entries are $W_{s+k(n+i+j)}^r$, 
where $i,j=0,1,\dots,r$, and $s$ and $k$ are any integers. This matrix is

\begin{equation} \label{eq30}
\mathbb{A}_n^{s,k}(r)=\begin{pmatrix}
W_{s+kn}^r & W_{s+k(n+1)}^{r} & \cdots &   W_{s+k(n+r)}^r\\
W_{s+k(n+1)}^r & W_{s+k(n+2)}^{r} & \cdots & W_{s+k(n+r+1)}^r\\
\vdots & \vdots & \ddots &  \vdots  \\
W_{s+k(n+r)}^r & W_{s+k(n+r+1)}^{r} & \cdots & W_{s+k(n+2r)}^r\\
\end{pmatrix}.
\end{equation}

We begin with the following proposition on the determinant whose entries 
are some power of linear polynomials. 

\begin{lemma}\label{main_lemma} Let $c_0,\dots,c_r$ and $x_0,\dots,x_r$ be real numbers. Then

 \begin{equation} \label{TT}
   \det( (c_jx_i+1)^r)_{0 \leq i,j \leq r}  = \prod_{0 \leq i < j \leq r}(x_i-x_j)
\prod_{0 \leq i < j \leq r}(c_i-c_j) \prod_{i=0}^r \binom{r}{i}.   
\end{equation}
\end{lemma}

\begin{proof}
We prove Lemma \ref{main_lemma} by using the factorization method \cite{krat}. The determinant will be $0$ if $x_0$ is replaced by any $x_i$ for $0<i\le r$, since some two rows of the matrix would be equal. This implies that $(x_0-x_i)$ is a factor of the determinant for each $i=1,2,\dots,r$. Similarly, we have that $(x_1-x_i)$ is a factor of the determinant for each $i=2,\dots,r$, and so on.  In a similar manner, we see that if $c_0$ is replaced by any $c_j$ for $0<j \le r$, then two columns of the matrix will be the same yielding the zero determinant. This implies that $(c_0-c_j)$ is a factor of the determinant for each $j=1,2,\dots,r$. Similarly, we have that $(c_1-c_j)$ is a factor of the determinant for each $j=2,\dots,r$, and so on. Therefore
\begin{equation}\label{factor}
\prod_{0 \leq i < j \leq r}(x_i-x_j)
\prod_{0 \leq i < j \leq r}(c_i-c_j)
\end{equation} 
is a factor of this determinant. As a function of $x_i$ for some fixed $i$ or a function of $c_j$ for some fixed $j$ the determinant is a polynomial of degree $r$. This implies that the factor \eqref{factor} and the required determinant have the same degree. Therefore, we can write the determinant as
\begin{equation*}
\det( (c_jx_i+1)^r)_{0 \leq i,j \leq r}  = C \prod_{0 \leq i < j \leq r}(x_i-x_j)
\prod_{0 \leq i < j \leq r}(c_i-c_j),
\end{equation*}
for some constant $C$.
To find $C$, we compare on both sides the coefficients of the monomial
\begin{equation}\label{monomial}
(c_rx_r)^r(c_{r-1}x_{r-1})^{r-1}\cdots(c_0x_0)^0.
\end{equation}
On the right-hand side, the monomial \eqref{monomial} appears as $$\prod_{0<i< j\le r}(-x_j)(-c_j) = \prod_{0<i<j\le r}x_jc_j.$$ Hence, the coefficient of the monomial \eqref{monomial} on the right-hand side is just equal to $C$. We see that for each $0\le i \le r$, the term $(c_ix_i)^i$ appears in $(c_ix_i+1)^r$. Hence, on the left-hand side, the monomial \eqref{monomial} appears as $$\pm\prod_{0\le i \le r}(c_ix_i+1)^r.$$ By the definition of the determinant, the sign in front of the expression is determined by the parity of the identity permutation $(0)(1)\cdots(r)$. Since $(0)(1)\cdots(r)=(01)(01)$, it follows that the identity permutation is even. Hence, the sign is determined to be $+$.  Since, for each $0\le i \le r$, the coefficient of $(c_ix_i)^i$ in $(c_ix_i+1)^r$ is $\binom{r}{i}$,  it follows that 
\[  C =  \prod_{i=0}^r \binom{r}{i}.\]   
This completes the proof of Lemma \ref{main_lemma}. 
\end{proof}

\begin{corollary}
Let $A(j), B(j), X_i, Y_i$ be real numbers for $0 \leq i,j \leq r$. Then 
\begin{multline}\label{restate-lemma}
\det( (A(j)X_i+B(j)Y_i)^r)_{0 \leq i,j \leq r} \\ = \prod_{0 \leq i < j \leq r}(X_iY_j-X_jY_i)
\prod_{0 \leq i < j \leq r}(A(i)B(j)-A(j)B(i)) \prod_{i=0}^r \binom{r}{i}. 
\end{multline}
\end{corollary}

\begin{proof}
We prove in the case of $B(j) \neq 0$ and $Y_i \neq 0$ for all  $0 \leq i,j \leq r$. Applying Lemma \ref{main_lemma} with $c_j=A(j)/B(j)$ and $x_i =X_i/Y_i$ for all  $0 \leq i,j \leq r$ and clearing the denominators, we obtain \eqref{restate-lemma}.  The proof of the other case when some of $B(j)$ or $Y_i$ are $0$ follows from the fact that the determinant with polynomial entries is a continuous function.
\end{proof}

Thus, this allows us to prove one of the main results of this paper.

\begin{theorem} \label{genfibo}
The determinant of the matrix $\mathbb{A}_n^{s,k}(r)$ is given by

\[ \det \mathbb{A}_n^{s,k}(r)= (-1)^{(s+kn+1) {r+1 \choose 2}}
\cdot c_2^{(s+kn) {r+1 \choose 2} +2k {r+1 \choose 3}}
\cdot \Delta^{r+1 \choose 2}
\cdot \prod_{i=0}^r {r\choose i}U_{(i+1)k}^{2(r-i)}.\]
\end{theorem}

\begin{proof}  By  \eqref{eq30}, \eqref{eqWWW}, \eqref{restate-lemma},  
and  \eqref{eqWU}  respectively, we have
\begin{align*}
& \det \mathbb{A}_n^{s,k}(r) = \det( W_{s+k(n+i+j)}^r)_{0 \leq i,j \leq r}    
=  \det( \left( A(j)W_{s+k(n+i)}+B(j)U_{s+k(n+r+i)} \right)^r)_{0 \leq i,j \leq r}   \\
&=  \prod_{0 \leq i < j \leq r}(W_{s+k(n+i)}U_{s+k(n+r+j)}-W_{s+k(n+j)}U_{s+k(n+r+i)})
\prod_{0 \leq i < j \leq r}\left(A(i)B(j)-A(j)B(i)\right)\prod_{i=0}^r \binom{r}{i} \\
& =    \prod_{0 \leq i < j \leq r}(  (-c_2)^{s+k(n+r+i)}U_{k(j-i)}W_{-kr})  
\prod_{0 \leq i < j \leq r}\left(\dfrac{-\Delta \cdot (-c_2)^{-kr}}{W_{-kr}^2}
(-c_2)^{ki}U_{k(j-i)}W_{-kr}\right)
\prod_{i=0}^r \binom{r}{i} \\
& = \prod_{0 \leq i < j \leq r}\left(  -\Delta \cdot (-c_2)^{s+k(n+2i)}U_{k(j-i)}^2\right)
\prod_{i=0}^r \binom{r}{i},  \\
\end{align*}
 
Rearranging the last expression allows us to obtain the desired identity. \end{proof}

\begin{remark}
By letting $W_n= F_n$ in Theorem \ref{genfibo} and noting that for the Fibonacci sequence $c_2 = 1$ and $\Delta = 1$, we then derive \eqref{eq2}. 
\end{remark}

\section{Determinants involving products of terms of second-order recurrence}
\label{sec:determinant-product}

The following lemma was mentioned by Krattenthaler \cite{krat} as part of the factorization method. We provide a different proof of this lemma using the Desnanot-Jacobi identity \cite{dodgson}.

\begin{lemma}\label{krat3}
Let $X_0,\dots,X_r, D_1,\dots,D_r$, and $E_1, \dots, E_r$ be indeterminates. Then 
\[  \det\left( \prod_{\ell=j+1}^{r}(X_i+D_\ell) 
\cdot \prod_{m=1}^{j}(X_i+E_m) \right)_{0 \leq i,j \leq r}  
=  \prod_{0 \leq i < j \leq r}(X_j-X_i) \cdot \prod_{1 \leq i \leq j \leq r}(D_j-E_i).  \]
\end{lemma}

An alternative way of writing this identity would be
\begin{multline}\label{det1}
\det\left( \prod_{\ell=j+1}^{r}(A(d_{\ell})X_i+B(e_{\ell})Y_i)\cdot 
\prod_{m=1}^{j}(A(e_m)X_i+B(e_m)Y_i)\right)_{0 \leq i,j \leq r} \\
= \prod_{0 \leq i < j \leq r}(X_iY_j-X_jY_i)
\prod_{1 \leq i \leq j \leq r}(B(e_i)A(d_j)-A(e_i)B(d_j)).
\end{multline}

The main result of this section is as follows:

\begin{theorem}\label{MyPow}
Let $d_1,\dots,d_r$ and $e_1, \dots, e_r$ be sequences of integers. Then
\begin{multline*} 
\det\left( \prod_{\ell=j+1}^{r}W_{s+k(n+i+d_\ell)} 
\cdot \prod_{m=1}^{j}W_{s+k(n+i+e_m)} \right)_{0 \leq i,j \leq r}  \\
= (-\Delta)^{{r+1 \choose 2}} 
\cdot (-c_2)^{(s+kn){r+1 \choose 2}+k{r+1 \choose 3}}
\cdot \prod_{\ell=1}^r U_{k\ell }^{r+1-\ell}
\prod_{1 \leq i \leq j \leq r}(-c_2)^{kd_j}U_{k(e_i-d_j)}.
\end{multline*}
\end{theorem}

\begin{proof}
We respectively apply the identities \eqref{eqWWW},\eqref{det1}, and \eqref{eqWW}, and the details of the proof are similar to those of Theorem \ref{genfibo}.
\end{proof}

The following slight variation of the result by Alfred \cite{alfred} called
\textit{basic power determinant} is a special case of this theorem.
\begin{corollary}\label{Alfred}
Let $r,s$, and $k$ be integers with $r\ge 0$. Then 
$$\begin{vmatrix}
F_{s}^r & F_{s}^{r-1}F_{s+k} & \cdots &   F_{s+k}^r\\
F_{s+k}^r & F_{s+k}^{r-1}F_{s+2k} & \cdots &   F_{s+2k}^r\\
\vdots & \vdots & \ddots &  \vdots  \\
F_{s+rk}^r & F_{s+rk}^{r-1}F_{s+(r+1)k} & \cdots &   F_{s+(r+1)k}^r\\
\end{vmatrix}
=(-1)^{(s+1){r+1 \choose 2}+k{r+1 \choose 3}}
F_k^{r+1 \choose 2}
\prod_{\ell=1}^r F_{k\ell }^{r+1-\ell}.$$

\end{corollary}

\begin{proof} This identity follows immediately from Theorem \ref{MyPow} by letting $d_1=d_2=\dots=d_r=0$,
$e_1=e_2=\dots=e_r=1$, $n=0$ together with the recurrence and the initial values of the Fibonacci numbers.
\end{proof}

Another interesting case arises when we let $(d_j)$ and $(e_j)$  in Theorem \ref{MyPow} be in some specific forms.

\begin{corollary}\label{thmN}
Let $s,k,n$, and $p$ be integers. Let $d_j = p-1+j$ and $e_j = j-1$ for $1 \leq j \leq r$. Then

\begin{multline*} \det\left( \prod_{\ell=j+1}^{r}W_{s+k(n+i+d_\ell)} 
\cdot \prod_{m=1}^{j}W_{s+k(n+i+e_m)}\right)_{0\le i,j \le r} \\
= \Delta^{r+1\choose 2}(-c_2)^{(s+kn){r+1 \choose 2}+2k{r+1\choose 3}} 
\cdot \prod_{\ell=1}^r U_{\ell k}^{r+1-\ell}
\cdot \prod_{\ell=0}^{r-1} U_{k(p+\ell)}^{r-\ell}.
\end{multline*}
\end{corollary}

\section{Acknowledgments}
We thank the anonymous referee and Ramesh Boonratana for their meticulous reading and helpful comments.

\bigskip
\hrule
\bigskip

\noindent 2010 {\it Mathematics Subject Classification}: Primary 11B39; Secondary 11B20. \\
\noindent \emph{Keywords:} Fibonacci number, determinant, second-order recurrence, Catalan identity.

\bigskip
\hrule
\bigskip

\noindent (Concerned with sequence
\seqnum{A000045}.)

\bigskip
\hrule
\bigskip


\bigskip
\hrule
\bigskip

\noindent
Return to
\htmladdnormallink{Journal of Integer Sequences home page}{http://www.cs.uwaterloo.ca/journals/JIS/}.
\vskip .1in

\end{document}